\documentclass[11pt]{amsart}
\usepackage{amsmath,amssymb,verbatim}
\usepackage[utf8]{inputenc}
\usepackage{enumerate}
\usepackage{pdflscape}
\usepackage{amsthm}
\usepackage{mathrsfs}
\usepackage{color}
\usepackage[normalem]{ulem}
\usepackage{cancel}
\usepackage{tikz}
\usetikzlibrary{matrix}
\usepackage[all]{xy}
\usepackage{mathtools}
\usepackage{cleveref}
\usepackage{arydshln}
\usepackage[shortlabels]{enumitem}

\makeatletter
\@namedef{subjclassname@1991}{\textup{2020} Mathematics Subject Classification}
\makeatother

\topskip=\baselineskip 
\newtheorem*{theorem*}{Theorem}

\newtheorem*{proposition*}{Proposition}

\newtheorem*{corollary*}{Corollary}

\newcommand{\GEN}[1]{\left\langle #1 \right\rangle}

\newcommand{\xt}{\widetilde{x}}
\newcommand{\yt}{\widetilde{y}}
\newcommand{\zt}{\widetilde{z}}

\newcommand{\qand}{\quad \text{and} \quad}

\title[	]{Non-isomorphic $2$-groups with \\ isomorphic modular group algebras}

\author{Diego Garc\'{\i}a-Lucas, Leo Margolis and \'{A}ngel del R\'{\i}o}
\thanks{This research is partially supported by Grant 19880/GERM/15 funded by Fundaci\'{o}n S\'{e}neca of Murcia, and Grant PID2020-113206GB-I00 funded by MCIN/AEI/10.13039/501100011033.}

\address{D. García-Lucas, \'{A}. del R\'{i}o: Departamento de Matem\'{a}ticas, Universidad de Murcia, 30100, Murcia, Spain. \rm{diego.garcial@um.es, adelrio@um.es}}

\address{L. Margolis: Vrije Universiteit Brussel, Department Wiskunde, Pleinlaan 2, 1040 Brussels, Belgium, \rm{leo.margolis@vub.be}}



\begin{document}
\maketitle

\begin{abstract}
We provide non-isomorphic finite $2$-groups which have isomorphic group algebras over any field of characteristic $2$, thus settling the Modular Isomorphism Problem. 
\end{abstract}

\section{Introduction}
A classical question, stated in the influential survey by Richard Brauer \cite{Bra63}, asks: ``When do nonisomorphic groups have isomorphic group algebras?'' Denoting by $RG$ the group ring of a group $G$ over a ring $R$, the more general question for group rings asks:

\begin{quote}
{\bf The Isomorphism Problem}: Does a ring isomorphism between $RG$ and $RH$ imply a group isomorphism between $G$ and $H$?
\end{quote}

The most classical version of the Isomorphism Problem which remained open until now, also mentioned in \cite[p. 166]{Bra63}, is:

\begin{quote}
{\bf The Modular Isomorphism Problem}: Let $G$ and $H$ be finite $p$-groups and $k$ a field of characteristic $p$. Does an isomorphism $kG\cong kH$ of rings imply an   isomorphism $G\cong H$ of groups?
\end{quote}

Here we show:

\begin{theorem*}
  There are non-isomorphic finite 2-groups $G$ and $H$ such that the group rings of $G$ and $H$ over any field of characteristic 2 are isomorphic. In particular, the Modular Isomorphism Problem has a negative answer.
\end{theorem*}

This also solves another old open problem (see e.g. \cite[p.258]{Scott1990}, \cite[Question 1.1]{NavarroSambale}, \cite[Question 4.11]{Linckelmann2018}):

\begin{corollary*}
  There are isomorphic 2-blocks of finite groups with non-isomorphic defect groups. In particular, the defect group of a 2-block is not determined by its Morita equivalence class over a finite field of characteristic 2.
\end{corollary*}

This puts an end to a long and rich history. It is clear that the complex group ring of a finite group only determines the multiset of irreducible character degrees, so for example all abelian groups of the same order have isomorphic complex group rings. While for $R$ the field of rationals the Isomorphism Problem has a positive solution for abelian groups \cite{PW50}, the two non-abelian groups of order $p^3$ for $p$ an odd prime have isomorphic group rings over the rationals.
 Much more interesting is the case of integral group rings. A positive solution for the Isomorphism Problem for finite abelian groups and $R$ the ring of integers appeared already in the thesis of G.~Higman \cite{HigmanThesis,Hig40}. 
Many other positive results have been proven for concrete classes of finite groups including metabelian or supersolvable groups, for $R$ the ring of integers \cite{Whitcomb,KimmerleHabil}, or $p$-groups, for $R$ the ring of integers of a $p$-adic field \cite{RoggenkampScott1987, Weiss1988, Roggenkamp92}.

However, some striking negative answers showed that even though group rings carry significant information about the source group, for most classes of groups and rings the Isomorphism Problem has a negative solution. For example, D.~Passman showed that for every prime $p$ there are at least $p^{2(n^3-23n^2)/27}$ non-isomorphic groups of order $p^n$ with isomorphic group algebras over every field of characteristic different from $p$ \cite{Passman65}. Moreover, E.~Dade exhibited a series   of pairs of metabelian non-isomorphic finite groups with isomorphic group algebras over all coefficient fields \cite{Dade71}. The latter solved Problem 2* in Brauer's survey, while the former showed the relevance of the characteristic of the coefficient ring. 
After these two negative results Passman wrote in 1977: ``There are, however, two glimmers of hope. The first concerns integral group rings, and the second concerns $p$-groups over $GF(p)$'' \cite[Page~664]{Pas77}. 
The first ``glimmer of hope'' faded away when M.~Hertweck found two non-isomorphic solvable groups with isomorphic group rings over the integers, and hence with isomorphic group rings over every ring \cite{Her01}. The second ``glimmer of hope'' was the Modular Isomorphism Problem. Note that the groups of Hertweck and Dade mentioned above have order divisible by two different primes. 

The study of modular group algebras of finite $p$-groups can be traced back at least to S.~Jennings who gave a group-theoretical description of the dimension subgroups \cite{Jen41}. The Modular Isomorphism Problem received considerable attention and positive solutions include, but are not limited to, groups which are abelian \cite{Deskins1956}, have trivial third dimension subgroup \cite{PS72}, $2$-groups of maximal class \cite{Carlson}, metacyclic \cite{BaginskiMetacyclic,San96}, of class 2 with elementary abelian derived subgroup \cite{San89}, or (elementary abelian)-by-cyclic \cite{Bag99}, as well as for groups of order $p^{n}$ with $n\le 5$ \cite{Passman1965p4, Makasikis, SalimSandlingp5} or $2^n$ with $n\le 8$ \cite{Wursthorn1993,BKRW99,Eick08,MM20}. 
However the Modular Isomorphism Problem resisted a solution so far. 
Surveys on the problem and related questions include \cite{San84, Bov98}, while it was also discussed in \cite{Bov74, Pas77, Seh78}.
Overviews of results known at the time are included in \cite{HS06,EK11}. More recent results appeared in \cite{BK19, Sakurai, BdR20, MargolisStanojkovski} 

In many ways our examples of $2$-groups with non-isomorphic group rings over any field of characteristic $2$ are minimal as they lie close to several classes of groups for which the Modular Isomorphism Problem is known to have a positive answer: Our groups are $2$-generated groups of nilpotency class $3$ with a derived subgroup which is cyclic of order $4$ and their orders take all the values $2^t$ for $t \ge 9$. Now the Modular Isomorphism Problem is known to have a positive answer for metacyclic groups \cite{BaginskiMetacyclic, San96}, for $2$-generated groups of nilpotency class $2$ \cite{BdR20}, and for $2$-generated groups of nilpotency class $3$ with elementary abelian derived subgroup \cite{MM20} (based on \cite{San89, Bag99}). Using an algorithm developed in \cite{Eick08} it was also shown to hold for groups of order $2^8$ \cite{MM20}. In \cite{EK11} the same algorithm was used to claim a positive answer also for groups of order $2^9$, but as pointed out in \cite{MM20} the programs used in the proof contained a flaw. It is worth to mention that our groups were found without the use of computers; the inspirations was theoretical and guided by the classifications of two-generated cyclic-by-abelian groups in terms of numerical invariants in \cite{OsnelDiegoAngel}. Still, the LAGUNA package of GAP \cite{LAGUNA} was very useful to perform calculations inside the group rings.

The solution of the Modular Isomorphism Problem also has consequences for Modular Representation Theory, particularly for the Theory of Blocks, such as the corollary given above. Together with the result mentioned above \cite{Roggenkamp92}, it also implies that Morita equivalence of $p$-blocks does not imply Morita equivalence over the ring of integers of a $p$-adic splitting field of the block, a question mentioned in \cite[Question 4.9]{Linckelmann2018} and \cite[p.94]{Craven}. Other applications could follow in particular from an analysis of our examples.

The Modular Isomorphism Problem remains open for various interesting classes of $p$-groups, including groups of nilpotency class 2 and groups of odd order.

\section{The groups}

We use standard group theoretical notation as for example $g^h = h^{-1}gh$ and $[g,h]=g^{-1}h^{-1}gh$, for $g$ and $h$ group elements. The Frattini subgroup of a group $G$ is denoted by $\Phi(G)$, the center by $Z(G)$, the derived subgroup by $G'$ and the centralizer in $G$ of a subset $X$ is denoted by $C_G(X)$.

Let $n$ and $m$ be integers satisfying $n>m>2$ and consider the groups given by the following presentations:
\begin{eqnarray*}
G&=&\GEN{x,y,z \ \mid \ z=[y,x], \ x^{2^n}=y^{2^m}=z^4=1, \ z^x=z^{-1}, \ z^y=z^{-1}}	\\
H&=&\GEN{a,b,c \  \mid \ c=[b,a], \ a^{2^n}=b^{2^m}=c^4=1, \ c^a=c^{-1}, \ c^b=c}
\end{eqnarray*}
The smallest admissible values, $n=4$ and $m=3$, correspond to the groups identified in the library of small groups of GAP \cite{GAP, SmallGroupLibrary} as [512, 456] and [512, 453], respectively.  Our theorem is a direct consequence of the following:

\begin{proposition*}
The groups $G$ and $H$ are non-isomorphic but if $k$ is a field of characteristic $2$ then the group algebras $kG$ and $kH$ are isomorphic.
\end{proposition*}

\begin{proof}
We start observing that the derived subgroups $G'=\GEN{z}$ and $H'=\GEN{c}$ are cyclic groups of order 4, $x^2\in Z(G)$, $a^2\in Z(H)$, $xy\in C_G(G')$ and $b\in C_H(H')$. Then $C_G(G')=\GEN{z,x^2,xy}$ and $C_H(H')=\GEN{c,a^2,b}$ and these groups are abelian. As $x$ has order $2^n$ and $y^2$ has order at most $2^{n-2}$, from $(xy)^2 = x^2y^2z^{-1}$ it follows that $C_G(G')$ has exponent $2^n$. On the other hand $a^2$ has order $2^{n-1}$ and $b$ has order at most $2^{n-1}$, hence the exponent of $C_H(H')$ is $2^{n-1}$. Therefore, the exponents of $C_G(G')$ and $C_H(H')$ are different and, in particular, $G$ and $H$ are not isomorphic.
		
If $F$ is the prime field of $k$ then $kG\cong k\otimes_F FG$ and therefore to prove that $kG\cong kH$ we may assume, without loss of generality, that $k$ is the field with 2 elements. 
We will prove that $kG$ and $kH$ are isomorphic by identifying in $kH$ a group basis isomorphic to $G$. 
In $kH$, we set 
	$$\widetilde{G}=\GEN{\xt,\yt} \text{ where} \quad  \xt=a \qand \yt=b(a+b+ab)c.$$
We will verify that $\widetilde{G}\cong G$ and $\widetilde{G}$ is a basis of $kH$. 
We use the following notation: 
$$A=a+1, \quad B=b+1, \quad C=c+1 \qand Y=\yt+1.$$

Let $I$ denote the augmentation ideal of $kH$. Note that this ideal is nilpotent.
It follows from Jennings' Theorem \cite[Theorem III.1.22]{Seh78} that $H/\Phi(H)$ is isomorphic to the additive group of $I/I^2$ and that $A+I^2$ and $B+I^2$ form a basis of $I/I^2$. 
By \cite[4.1.3]{KP69}, the intersection of the maximal right ideals of $I$ equals $I^2$.
Observe that $C=b^{-1}a^{-1}(ba-ab)=b^{-1}a^{-1}(BA-AB)\in I^2$. 
 Hence $c\equiv 1 \mod I^2$. It follows that $\yt\equiv b(a+b+ab)\mod I^2$, so $b+\yt\equiv b(1+a)(1+b)\equiv 0 \mod I^2$. Thus $\yt\equiv b\mod I^2$.
 
Therefore the additive subgroup of $I/I^2$ is generated by $A+I^2$ and $Y+I^2$ and hence, by the Burnside Basis Theorem \cite[Theorem~4.1.4]{KP69}, we have that $A$ and $Y$ generate $I$ as a ring. 

We now verify that $\xt, \yt$ and $\zt=[\yt,\xt]$ satisfy the relations of $x$ and $y$ in the presentation of $G$ given above. Of course, $\xt^{2^n}=1$  and $\xt^2=a^2 \in Z(kH)$. 
The latter implies that $1=[\yt,\xt^2]=\zt \; \zt^{\xt}$, so that $\zt^{\xt}=\zt^{-1}$. 
Moreover, $b^2c$ belongs to the center of $H$.  
Thus $\yt=b^2c+ba(1+b)c$ is a sum of two commuting elements, and therefore 
\begin{eqnarray*}
\yt^2 &=& (b^2c)^2 + (ba(1+b)c)^2 = (b^2c)^2 + (ab(1+b)c^2)^2 = 
(b^2c)^2 + a^2b^2c(1+bc)(1+b) \\ 
&=& (b^2c)^2 + a^2(b^2c) + a^2(b^2c)^2 + a^2(b^2c)(b+bc) \in Z(kH)
\end{eqnarray*}
The latter is a sum of products of central elements of $kH$ because $a^2$ belongs to the center of $H$ and $\{b,bc\}$ is the conjugacy class of $b$ in $H$.    
Then we have $\zt^{\yt}=\zt^{-1}$, as above, and 
$$\yt^{2^m}=b^{2^{m+1}}c^{2^m}+a^{2^m}(b^{2^m}c^{2^{m-1}}+b^{2^{m+1}}c^{2^m})+a^{2^m}b^{2^m}c^{2^{m-1}}(b^{2^{m-1}}+b^{2^{m-1}}c^{2^{m-1}})=1.$$
Finally, let $J$ be the ideal of $kH$ generated by $C$. Then $kH/J$ is commutative and, as $\zt$ is a commutator in the unit group of $kH$ it follows that $\zt\in 1+J$. Then $1+\zt\in J$ and as $J^4=0$ it follows that $\zt^4=1$. 

We conclude that $\widetilde{G}$ is an epimorphic image of $G$ and $\widetilde{G}$ includes a basis of $kH$ as a $k$-vector space. 
As $|H|=|G|$ it follows that $\widetilde{G}\cong G$. 
\end{proof}


\textbf{Acknowledgments}: The second author would like to thank Mima Stanojkovski for many helpful conversations on the Modular Isomorphism Problem.

\bibliographystyle{amsalpha}
\bibliography{MIP}

\providecommand{\bysame}{\leavevmode\hbox to3em{\hrulefill}\thinspace}
\providecommand{\MR}{\relax\ifhmode\unskip\space\fi MR }
\providecommand{\MRhref}[2]{%
  \href{http://www.ams.org/mathscinet-getitem?mr=#1}{#2}
}
\providecommand{\href}[2]{#2}
\begin{thebibliography}{BGLdR21}

\bibitem[Bag88]{BaginskiMetacyclic}
C.~Bagi\'{n}ski, \emph{The isomorphism question for modular group algebras of
  metacyclic {$p$}-groups}, Proc. Amer. Math. Soc. \textbf{104} (1988), no.~1,
  39--42.

\bibitem[Bag99]{Bag99}
\bysame, \emph{On the isomorphism problem for modular group algebras of
  elementary abelian-by-cyclic {$p$}-groups}, Colloq. Math. \textbf{82} (1999),
  no.~1, 125--136.

\bibitem[BdR21]{BdR20}
O.~Broche and \'{A}. del R\'{\i}o, \emph{The {M}odular {I}somorphism {P}roblem
  for two generated groups of class two}, Indian Journal of Pure and Applied
  Mathematics, in press (2021), 1--8,
  https://doi.org/10.1007/s13226-021-00182-w, https://arxiv.org/abs/2003.13281.

\bibitem[BEO19]{SmallGroupLibrary}
H.~U. Besche, B~Eick, and E.~O'Brien, \emph{{SmallGrp}: The {GAP Small Groups
  Library}, version 1.4.1}, https://gap-packages.github.io/smallgrp/, 2019.

\bibitem[BGLdR21]{OsnelDiegoAngel}
O.~Broche, D.~Garc\'{\i}a-Lucas, and \'{A}. del R\'{\i}o, \emph{A
  classification of the finite two-generated cyclic-by-abelian groups of prime
  power order}, 1--25, http://arxiv.org/abs/2106.06449.

\bibitem[BK19]{BK19}
C.~Bagi\'{n}ski and J.~Kurdics, \emph{The modular group algebras of
  {$p$}-groups of maximal class {II}}, Comm. Algebra \textbf{47} (2019), no.~2,
  761--771.

\bibitem[BKRS19]{LAGUNA}
V.~Bovdi, A.~Konovalov, R.~Rossmanith, and C.~Schneider, \emph{{LAGUNA}: Lie
  algebras and units of group algebras, version 3.9.3},
  https://gap-packages.github.io/laguna/, 2019.

\bibitem[BKRW99]{BKRW99}
F.~M. Bleher, W.~Kimmerle, K.~W. Roggenkamp, and M.~Wursthorn,
  \emph{Computational aspects of the isomorphism problem}, Algorithmic algebra
  and number theory ({H}eidelberg, 1997), Springer, Berlin, 1999, pp.~313--329.

\bibitem[Bov74]{Bov74}
A.~A. Bovdi, \emph{Gruppovye kol'ca}, U\v{z}gorod. Gosudarstv. Univ., Uzhgorod,
  1974, A textbook.

\bibitem[Bov98]{Bov98}
\bysame, \emph{The group of units of a group algebra of characteristic {$p$}},
  Publ. Math. Debrecen \textbf{52} (1998), no.~1-2, 193--244.

\bibitem[Bra63]{Bra63}
R.~Brauer, \emph{Representations of finite groups}, Lectures on {M}odern
  {M}athematics, {V}ol. {I}, Wiley, New York, 1963, pp.~133--175.

\bibitem[Car77]{Carlson}
J.~F. Carlson, \emph{Periodic modules over modular group algebras}, J. London
  Math. Soc. (2) \textbf{15} (1977), no.~3, 431--436.

\bibitem[Cra19]{Craven}
D.~A. Craven, \emph{Representation theory of finite groups: a guidebook},
  Universitext, Springer, Cham, 2019.

\bibitem[Dad71]{Dade71}
E.~Dade, \emph{Deux groupes finis distincts ayant la m\^{e}me alg\`ebre de
  groupe sur tout corps}, Math. Z. \textbf{119} (1971), 345--348.

\bibitem[Des56]{Deskins1956}
W.~E. Deskins, \emph{Finite {A}belian groups with isomorphic group algebras},
  Duke Math. J. \textbf{23} (1956), 35--40. \MR{77535}

\bibitem[Eic08]{Eick08}
B.~Eick, \emph{Computing automorphism groups and testing isomorphisms for
  modular group algebras}, J. Algebra \textbf{320} (2008), no.~11, 3895--3910.

\bibitem[EK11]{EK11}
B.~Eick and A.~Konovalov, \emph{The modular isomorphism problem for the groups
  of order 512}, Groups {S}t {A}ndrews 2009 in {B}ath. {V}olume 2, London Math.
  Soc. Lecture Note Ser., vol. 388, Cambridge Univ. Press, Cambridge, 2011,
  pp.~375--383.

\bibitem[GAP19]{GAP}
The GAP~Group, \emph{{GAP -- Groups, Algorithms, and Programming, Version
  4.10.2}}, 2019, http://www.gap-system.org.

\bibitem[Her01]{Her01}
M.~Hertweck, \emph{A counterexample to the isomorphism problem for integral
  group rings}, Ann. of Math. (2) \textbf{154} (2001), no.~1, 115--138.

\bibitem[Hig40a]{HigmanThesis}
G.~Higman, \emph{Units in group rings}, 1940, Thesis (Ph.D.)--Univ. Oxford.

\bibitem[Hig40b]{Hig40}
\bysame, \emph{The units of group-rings}, Proc. London Math. Soc. (2)
  \textbf{46} (1940), 231--248.

\bibitem[HS06]{HS06}
M.~Hertweck and M.~Soriano, \emph{On the modular isomorphism problem: groups of
  order {$2^6$}}, Groups, rings and algebras, Contemp. Math., vol. 420, Amer.
  Math. Soc., Providence, RI, 2006, pp.~177--213.

\bibitem[Jen41]{Jen41}
S.~A. Jennings, \emph{The structure of the group ring of a {$p$}-group over a
  modular field}, Trans. Amer. Math. Soc. \textbf{50} (1941), 175--185.

\bibitem[Kim91]{KimmerleHabil}
W.~Kimmerle, \emph{Beitr\"age zur ganzzahligen {D}arstellungstheorie endlicher
  {G}ruppen}, Bayreuth. Math. Schr. (1991), no.~36, 139.

\bibitem[KP69]{KP69}
R.~L. Kruse and D.~T. Price, \emph{Nilpotent rings}, Gordon and Breach Science
  Publishers, New York-London-Paris, 1969.

\bibitem[Lin18]{Linckelmann2018}
M.~Linckelmann, \emph{Finite-dimensional algebras arising as blocks of finite
  group algebras}, Representations of algebras, Contemp. Math., vol. 705, Amer.
  Math. Soc., [Providence], RI, [2018] \copyright 2018, pp.~155--188.

\bibitem[Mak76]{Makasikis}
A.~Makasikis, \emph{Sur l'isomorphie d'alg\`ebres de groupes sur un champ
  modulaire}, Bull. Soc. Math. Belg. \textbf{28} (1976), no.~2, 91--109.
  \MR{561324}

\bibitem[MM20]{MM20}
L.~{Margolis} and T.~{Moede}, \emph{{The Modular Isomorphism Problem for small
  groups -- revisiting Eick's algorithm}}, 1--10,
  https://arxiv.org/abs/2010.07030.

\bibitem[MS21]{MargolisStanojkovski}
L.~Margolis and M.~Stanojkovski, \emph{On the modular isomorphism problem for
  groups of class 3}, J. Group Theory, in press (2021), 1--44,
  https://doi.org/10.1515/jgth-2020-0174, arXiv.org/abs/2009.13970.

\bibitem[NS18]{NavarroSambale}
G.~Navarro and B.~Sambale, \emph{On the blockwise modular isomorphism problem},
  Manuscripta Math. \textbf{157} (2018), no.~1-2, 263--278.

\bibitem[Pas65a]{Passman1965p4}
D.~S. Passman, \emph{The group algebras of groups of order {$p^{4}$} over a
  modular field}, Michigan Math. J. \textbf{12} (1965), 405--415. \MR{0185022}

\bibitem[Pas65b]{Passman65}
\bysame, \emph{Isomorphic groups and group rings}, Pacific J. Math. \textbf{15}
  (1965), 561--583.

\bibitem[Pas77]{Pas77}
\bysame, \emph{The algebraic structure of group rings}, Pure and Applied
  Mathematics, Wiley-Interscience [John Wiley \& Sons], New York-London-Sydney,
  1977.

\bibitem[PS72]{PS72}
I.~B.~S. Passi and S.~K. Sehgal, \emph{Isomorphism of modular group algebras},
  Math. Z. \textbf{129} (1972), 65--73.

\bibitem[PW50]{PW50}
S.~Perlis and G.~L. Walker, \emph{Abelian group algebras of finite order},
  Trans. Amer. Math. Soc. \textbf{68} (1950), 420--426.

\bibitem[Rog92]{Roggenkamp92}
K.~W. Roggenkamp, \emph{Subgroup rigidity of {$p$}-adic group rings ({W}eiss
  arguments revisited)}, J. London Math. Soc. (2) \textbf{46} (1992), no.~3,
  432--448.

\bibitem[RS87]{RoggenkampScott1987}
K.~W. Roggenkamp and L.~Scott, \emph{Isomorphisms of {$p$}-adic group rings},
  Ann. of Math. (2) \textbf{126} (1987), no.~3, 593--647.

\bibitem[Sak20]{Sakurai}
T.~Sakurai, \emph{The isomorphism problem for group algebras: a criterion}, J.
  Group Theory \textbf{23} (2020), no.~3, 435--445.

\bibitem[San85]{San84}
R.~Sandling, \emph{The isomorphism problem for group rings: a survey}, Orders
  and their applications ({O}berwolfach, 1984), Lecture Notes in Math., vol.
  1142, Springer, Berlin, 1985, pp.~256--288.

\bibitem[San89]{San89}
\bysame, \emph{The modular group algebra of a central-elementary-by-abelian
  {$p$}-group}, Arch. Math. (Basel) \textbf{52} (1989), no.~1, 22--27.

\bibitem[San96]{San96}
\bysame, \emph{The modular group algebra problem for metacyclic {$p$}-groups},
  Proc. Amer. Math. Soc. \textbf{124} (1996), no.~5, 1347--1350.

\bibitem[Sco90]{Scott1990}
L.~Scott, \emph{Defect groups and the isomorphism problem}, Ast\'{e}risque
  (1990), no.~181-182, 257--262.

\bibitem[Seh78]{Seh78}
S.~K. Sehgal, \emph{Topics in group rings}, Monographs and Textbooks in Pure
  and Applied Math., vol.~50, Marcel Dekker, Inc., New York, 1978.

\bibitem[SS96]{SalimSandlingp5}
M.~A.~M. Salim and R.~Sandling, \emph{The modular group algebra problem for
  groups of order {$p^5$}}, J. Austral. Math. Soc. Ser. A \textbf{61} (1996),
  no.~2, 229--237.

\bibitem[Wei88]{Weiss1988}
A.~Weiss, \emph{Rigidity of {$p$}-adic {$p$}-torsion}, Ann. of Math. (2)
  \textbf{127} (1988), no.~2, 317--332.

\bibitem[Whi68]{Whitcomb}
A.~Whitcomb, \emph{The {G}roup {R}ing {P}roblem}, ProQuest LLC, Ann Arbor, MI,
  1968, Thesis (Ph.D.)--The University of Chicago.

\bibitem[Wur93]{Wursthorn1993}
M.~Wursthorn, \emph{Isomorphisms of modular group algebras: an algorithm and
  its application to groups of order {$2^6$}}, J. Symbolic Comput. \textbf{15}
  (1993), no.~2, 211--227. \MR{1218760}

\end{thebibliography}

\end{document}